\documentclass[11pt]{amsart}
\usepackage{graphicx,amscd,color}
\usepackage{a4wide}

\theoremstyle{plain}
\newtheorem{theorem}{Theorem}[section]
\newtheorem{proposition}[theorem]{Proposition}
\newtheorem{lemma}[theorem]{Lemma}

\newtheorem{conjecture}[theorem]{Conjecture}
\newtheorem{question}[theorem]{Question}

\theoremstyle{remark}

\newtheorem{claim}{Claim}

\makeatletter
\@namedef{subjclassname@2020}{\textup{2020} Mathematics Subject Classification}
\makeatother

\begin{document}

\title [Unperturbed bridge positions]
{Unperturbed weakly reducible non-minimal\\ bridge positions}

\author[J. H. Lee]{Jung Hoon Lee}
\address{Department of Mathematics and Institute of Pure and Applied Mathematics, Jeonbuk National University, Jeonju 54896, Korea}
\email{junghoon@jbnu.ac.kr}



\begin{abstract}
A bridge position of a knot is said to be perturbed if there exists a cancelling pair of bridge disks.
Motivated by the examples of knots admitting unperturbed strongly irreducible non-minimal bridge positions due to Jang-Kobayashi-Ozawa-Takao,
we derive examples of unperturbed weakly reducible non-minimal bridge positions.
Also, a bridge version of Gordon's Conjecture is proposed:
the connected sum of unperturbed bridge positions is unperturbed.
\end{abstract}

\maketitle

\section{Introduction}\label{sec1}

Suppose that $S^3$ is decomposed into two $3$-balls by an embedded sphere $S$.
A knot $K$ is in {\em $n$-bridge position} with respect to $S$ if
$K$ intersects each of the $3$-balls in a collection of $n$ $\partial$-parallel arcs.
The original concept of bridge position, the bridge number, was first introduced by Schubert in 1954 \cite{Schubert}.
Thereafter it is generalized to the notion of bridge splitting for (a $3$-manifold, a link) pair.

For any $n$-bridge position, we can always give a perturbation to get a perturbed $(n+1)$-bridge position.
Conversely, from a perturbed bridge position we obtain a lower index bridge position.
A bridge position is {\em unperturbed} if it is not perturbed.
It is a fundamental problem to detect whether a given bridge position is unperturbed or not.
The unknot has a unique $1$-bridge position and every $n$-bridge position ($n > 1$) of the unknot is perturbed \cite{Otal1}.
Non-minimal bridge positions of $2$-bridge knots \cite{Otal2}, torus knots \cite{Ozawa} are perturbed.
Zupan showed that if $K$ is an mp-small knot and every non-minimal bridge position of $K$ is perturbed,
then every non-minimal bridge position of a $(p, q)$-cable of $K$ is also perturbed \cite{Zupan}.
Concerning $2$-cables, the author showed that if every non-minimal bridge position of a knot $K$ is perturbed,
then every non-minimal bridge position of a $(2, 2q)$-cable link of $K$ is perturbed,
without the assumption of mp-smallness of $K$ \cite{Lee}.

On the other hand, there exist knots admitting unperturbed non-minimal bridge positions \cite{JKOT}, \cite{Ozawa-Takao}.
All the examples in \cite{JKOT} and \cite{Ozawa-Takao} are strongly irreducible bridge positions.
Weakly reducible bridge positions are the opposites of more complicated strongly irreducible ones, so simpler.
For weakly reducible bridge positions, one can ask whether unperturbed non-minimal bridge positions can be attained.
We show that there exist unperturbed weakly reducible non-minimal bridge positions
by taking the connected sum operation on the knots due to Jang et al.

\begin{theorem}\label{thm1}
There exist unperturbed weakly reducible non-minimal bridge positions.
\end{theorem}

The unperturbedness is shown by the method of $2$-fold branched covering and Gordon's Conjecture.
Can we prove it directly without taking a $2$-fold branched covering?
This raises the following conjecture.

\begin{conjecture}[A bridge version of Gordon's Conjecture]\label{conj1}
The connected sum of two unperturbed bridge positions is unperturbed.
\end{conjecture}

The presented examples of knots for Theorem \ref{thm1} are composite knots.
We have the following question.

\begin{question}\label{question1}
Does there exist a prime knot admitting an unperturbed weakly reducible non-minimal bridge position?
\end{question}

\section{Bridge positions}\label{sec2}

Let $B$ be a $3$-ball.
A {\em trivial tangle} is a collection of disjoint properly embedded arcs $b_1, \ldots, b_n$ in $B$ such that
each $b_i$ cobounds a disk $D_i$ with an arc in $\partial B$ and $D_i \cap b_j = \emptyset$ for all $j \ne i$.
Suppose that a $2$-sphere $S$ decomposes $S^3$ into two $3$-balls $B$ and $C$.
Let $K$ be a knot.
If $B \cap K$ and $C \cap K$ are trivial tangles, each consisting of $n$ arcs,
then we say that $K$ is in {\em $n$-bridge position} with respect to $S$.
Each arc of the trivial tangles $B \cap K$ and $C \cap K$ is called a {\em bridge}.
A bridge $b_i$ of the trivial tangle, say $B \cap K = \{ b_1, \ldots, b_n \}$,
cobounds a {\em bridge disk} $D_i$ with an arc in $S$ such that $D_i \cap b_j = \emptyset$ for all $j \ne i$ by definition.
By standard cut-and-paste argument, $D_i$'s ($i=1, \ldots, n$) can be taken to be pairwise disjoint.
A collection $\{ D_1,\ldots, D_n \}$ of $n$ disjoint bridge disks is called a {\em complete bridge disk system}.
If $K$ is in bridge position,
we have a decomposition of the pair $(S^3, K)$ into $(B, B \cap K)$ and $(C, C \cap K)$.
But when it is clear from the context, we will simply use the notation $B \cup_S C$ to indicate the bridge position.

For an $n$-bridge position $B \cup_S C$, we can perturb a small neighborhood of a point $p$ of $K \cap S$
so that it becomes an $(n+1)$-bridge position having bridge disks $D \subset B$ and $E \subset C$ with $D \cap E = p$.
Such an operation is called a {\em perturbation}, and
a bridge position isotopic to one obtained by a perturbation is said to be {\em perturbed}.
Each of $D$ and $E$ is a {\em cancelling disk} and $(D, E)$ is a {\em cancelling pair}.
Conversely, a perturbation can be reversed to give a lower index bridge position.
A bridge position is {\em unperturbed} if it is not perturbed.

A disk $D$ properly embedded in $B$ or $C$ with $D \cap K = \emptyset$ is a {\em compressing disk} if
$\partial D$ does not bound a disk in $S - K$.
A bridge position $B \cup_S C$ is {\em weakly reducible} if
there exist compressing disks $D \subset B$ and $E \subset C$ such that $D \cap E = \emptyset$.
Otherwise, it is {\em strongly irreducible}.
It is easy to see that if an $n$-bridge position ($n \ge 3$) is perturbed, then it is weakly reducible.
Note that a $2$-bridge position of the unknot is perturbed and strongly irreducible.

The {\em bridge number} $b(K)$ of a knot $K$ is the minimum of
$\{ n \,|\, K {\textrm{ admits an }} n{\textrm{-bridge position}} \}$.
For a connected sum $K_1 \# K_2$ of two knots $K_1$ and $K_2$,
$b(K_1 \# K_2)$ is $b(K_1) + b(K_2) -1$ \cite{Schubert}, \cite{Schultens1}.
For a $(p,q)$-torus knot $K_{p,q}$, $b(K_{p,q}) = \min \{ |p|, |q| \}$ \cite{Schubert}, \cite{Schultens2}.

\section{Heegaard splittings}\label{sec3}

In this section, we briefly review basic notions and facts about Heegaard splittings.
Connections between Heegaard splittings and bridge positions, via $2$-fold branched coverings,
will be discussed in the subsequent sections.

For a closed $3$-manifold $M$, a {\em Heegaard splitting} $V \cup_F W$ is
a decomposition of $M$ into two handlebodies $V$ and $W$ of the same genus.
The common boundary $F$ of $V$ and $W$ is called the {\em Heegaard surface} of $V \cup_F W$.
A Heegaard splitting $V \cup_F W$ is {\em stabilized} if
there exist disks $D \subset V$ and $E \subset W$ such that $| D \cap E | = 1$, and
$(D, E)$ is called a {\em cancelling pair}.
Otherwise, it is {\em unstabilized}.
If there exist compressing disks $D \subset V$ and $E \subset W$ such that
$\partial D = \partial E$ ($D \cap E = \emptyset$ respectively),
then the Heegaard splitting is said to be {\em reducible} ({\em weakly reducible} respectively).
A Heegaard splitting is {\em irreducible} ({\em strongly irreducible} respectively) if
it is not reducible (weakly reducible respectively).
It is immediate that a reducible Heegaard splitting is weakly reducible,
by slightly pushing one of $D$ and $E$ with $\partial D = \partial E$ to be apart from the other.

Suppose that $V \cup_F W$ is stabilized with a cancelling pair $(D, E)$ and the genus of $F$ is at least two.
Then we can see that $V \cup_F W$ is reducible, hence weakly reducible,
by band summing two copies of $D$ along $\partial E$ and band summing two copies of $E$ along $\partial D$.
As a contrapositive, we have the following.

\begin{proposition}\label{prop1}
If a Heegaard splitting of genus $g \ge 2$ is strongly irreducible, then it is unstabilized.
\end{proposition}

\section{$2$-Fold branched coverings}\label{sec4}

Let $B \cup_S C$ be an $n$-bridge position of a knot $K$.
Let $\{ D_1, \ldots, D_n \}$ be a complete bridge disk system for $B \cap K$.
Cut $B$ along $\bigcup^n_{i=1} D_i$.
Let $B'$ be the resulting $3$-ball and let $D'_{i,+}$ and $D'_{i,-}$ denote the two scars of $D_i$ on $\partial B'$.
Let $B''$ be a copy of $B'$ and similarly let $D''_{i,+}$ and $D''_{i,-}$ denote the two scars of $D_i$ on $\partial B''$.
Glue $B'$ and $B''$ along $D'_{i,\pm}$ and $D''_{i,\mp}$ for each $i$.
The resulting manifold is a genus $n-1$ handlebody $V$.
There is an involution of $V$ fixing $B \cap K$ such that
the quotient map induced by the involution is a $2$-fold covering $p_1: V \to B$ branched along $B \cap K$.
Similarly, we can take a $2$-fold covering $p_2: W \to C$ branched along $C \cap K$, where $W$ is a genus $n-1$ handlebody.
Hence we have a $2$-fold branched covering map $p$ from a genus $n-1$ Heegaard splitting $V \cup_F W$ to $B \cup_S C$,
branched along the knot $K$ in $n$-bridge position.

Suppose $B \cup_S C$ is perturbed, so it admits a cancelling pair $(D, E)$.
The preimages $p^{-1}(D)$ and $p^{-1}(E)$ are disks in $V$ and $W$ respectively that intersect at one point,
so $V \cup_F W$ is stabilized.
As a contrapositive, we have the following.

\begin{proposition}\label{prop2}
Suppose that $p: V \cup_F W \to B \cup_S C$ is
a $2$-fold covering branched along a knot $K$ in bridge position with respect to $S$.
If $V \cup_F W$ is unstabilized, then $B \cup_S C$ is unperturbed.
\end{proposition}

The converse of Proposition \ref{prop2} does not hold.
There is a relevant discussion in \cite[Section $1$]{Howards-Schultens}.
Let $K_{p,q}$ be a $(p, q)$-torus knot with $0 < p < q$.
A $p$-bridge position $B \cup_S C$ of $K_{p,q}$ is unperturbed since $b(K_{p,q}) = p$.

A $2$-fold covering of $S^3$ branched along $K_{p,q}$ is a small Seifert fibered manifold $M$.
It is known that an irreducible Heegaard splitting of a Seifert fibered manifold
is either vertical or horizontal \cite{Moriah-Schultens}.
The genus of a vertical splitting of $M$ is at most two.
The genus of a horizontal splitting is always an even number.
Refer to \cite{Moriah-Schultens} for more details.

The $2$-fold branched covering $V \cup_F W$ of $B \cup_S C$ is of genus $p - 1$.
So for example, if $(p,q)=(4,5)$, then $V \cup_F W$ is a reducible Heegaard splitting of $M$.
Since $M$ is an irreducible manifold, $V \cup_F W$ is stabilized.
Therefore, if $(p,q)=(4,5)$, then $B \cup_S C$ is unperturbed and $V \cup_F W$ is stabilized.

\section{Connected sums}\label{sec5}

Let $B_1 \cup_{S_1} C_1$ and $B_2 \cup_{S_2} C_2$ be bridge positions of knots $K_1$ and $K_2$ respectively.
Let $p_i: V_i \cup_{F_i} W_i \to B_i \cup_{S_i} C_i$ ($i=1,2$) be $2$-fold branched coverings explained in Section \ref{sec4}.
See Figure \ref{fig1}.
The connected sum of $B_1 \cup_{S_1} C_1$ and $B_2 \cup_{S_2} C_2$ is defined as follows.
Take a small open ball neighborhood $N_i$ at a point of $K_i \cap S_i$.
Glue $B_1 - N_1$ and $B_2 - N_2$ along $B_1 \cap \partial N_1$ and $B_2 \cap \partial N_2$ so that
$K_1 \cap (B_1 \cap \partial N_1)$ is identified with $K_2 \cap (B_2 \cap \partial N_2)$.
Similarly, glue $C_1 - N_1$ and $C_2 - N_2$ along $C_1 \cap \partial N_1$ and $C_2 \cap \partial N_2$ so that
$K_1 \cap (C_1 \cap \partial N_1)$ is identified with $K_2 \cap (C_2 \cap \partial N_2)$.
The result is a bridge position $(B_1 \natural B_2) \cup_{S_1 \# S_2} (C_1 \natural C_2)$ of $K_1 \# K_2$.

\begin{enumerate}
\item[(1)] $(B_1 \cup_{S_1} C_1) \# (B_2 \cup_{S_2} C_2) = (B_1 \natural B_2) \cup_{S_1 \# S_2} (C_1 \natural C_2).$
\end{enumerate}

Now we consider the connected sum of $M_1 = V_1 \cup_{F_1} W_1$ and $M_2 = V_2 \cup_{F_2} W_2$.
Since we want the connected sum to be compatible with the branched covering map,
take $p_1^{-1}(N_1)$ and $p_2^{-1}(N_2)$, which are open $3$-balls.
Glue $V_1 - p_1^{-1}(N_1)$ and $V_2 - p_2^{-1}(N_2)$ along
$V_1 \cap \partial (p_1^{-1}(N_1))$ and $V_2 \cap \partial (p_2^{-1}(N_2))$.
Similarly, glue $W_1 - p_1^{-1}(N_1)$ and $W_2 - p_2^{-1}(N_2)$ along
$W_1 \cap \partial (p_1^{-1}(N_1))$ and $W_2 \cap \partial (p_2^{-1}(N_2))$.
The result is a Heegaard splitting $(V_1 \natural V_2) \cup_{F_1 \# F_2} (W_1 \natural W_2)$ of $M_1 \# M_2$.

\begin{enumerate}
\item[(2)] $(V_1 \cup_{F_1} W_1) \# (V_2 \cup_{F_2} W_2) = (V_1 \natural V_2) \cup_{F_1 \# F_2} (W_1 \natural W_2).$
\end{enumerate}

\begin{figure}[!hbt]
\includegraphics[width=14cm,clip]{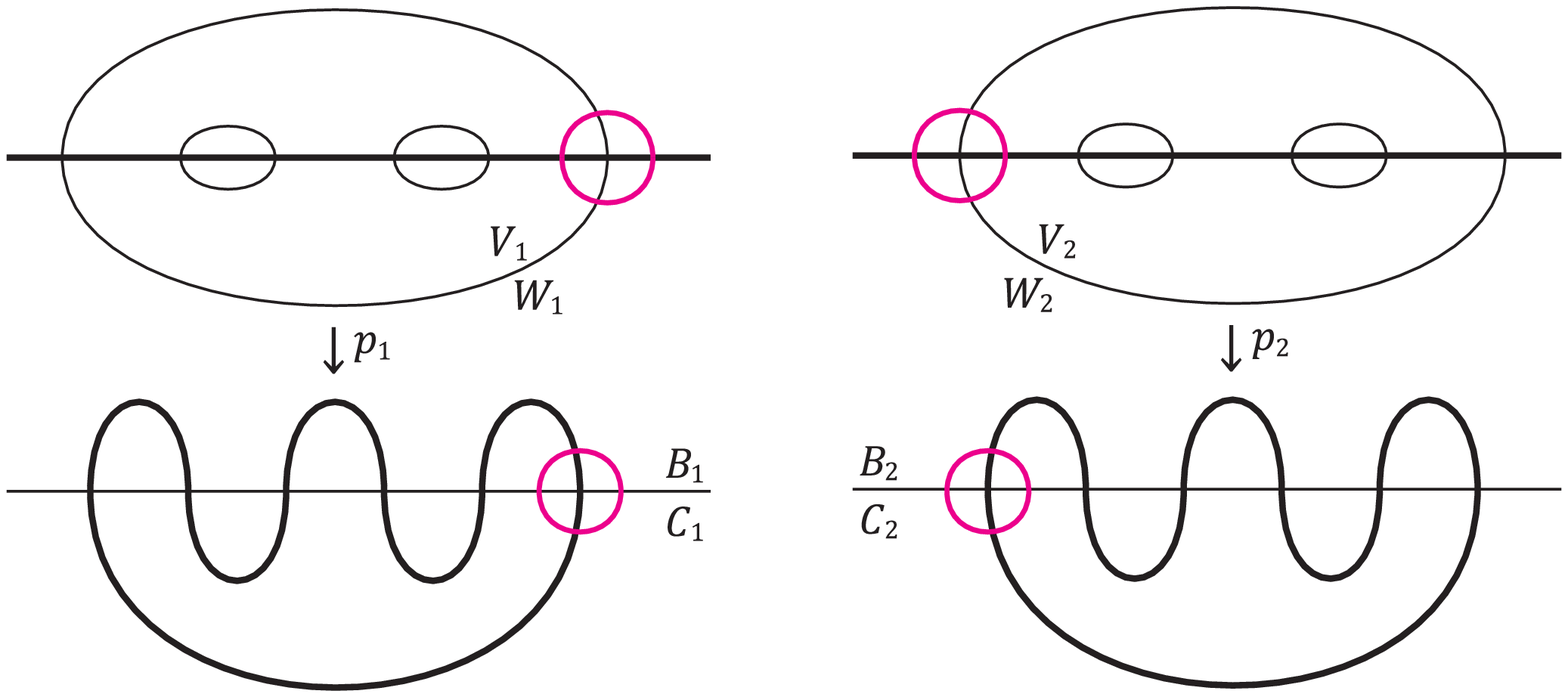}
\caption{$2$-Fold branched coverings $p_1: V_1 \cup_{F_1} W_1 \to B_1 \cup_{S_1} C_1$ and
$p_2: V_2 \cup_{F_2} W_2 \to B_2 \cup_{S_2} C_2$.}\label{fig1}
\end{figure}

\begin{figure}[!hbt]
\includegraphics[width=9cm,clip]{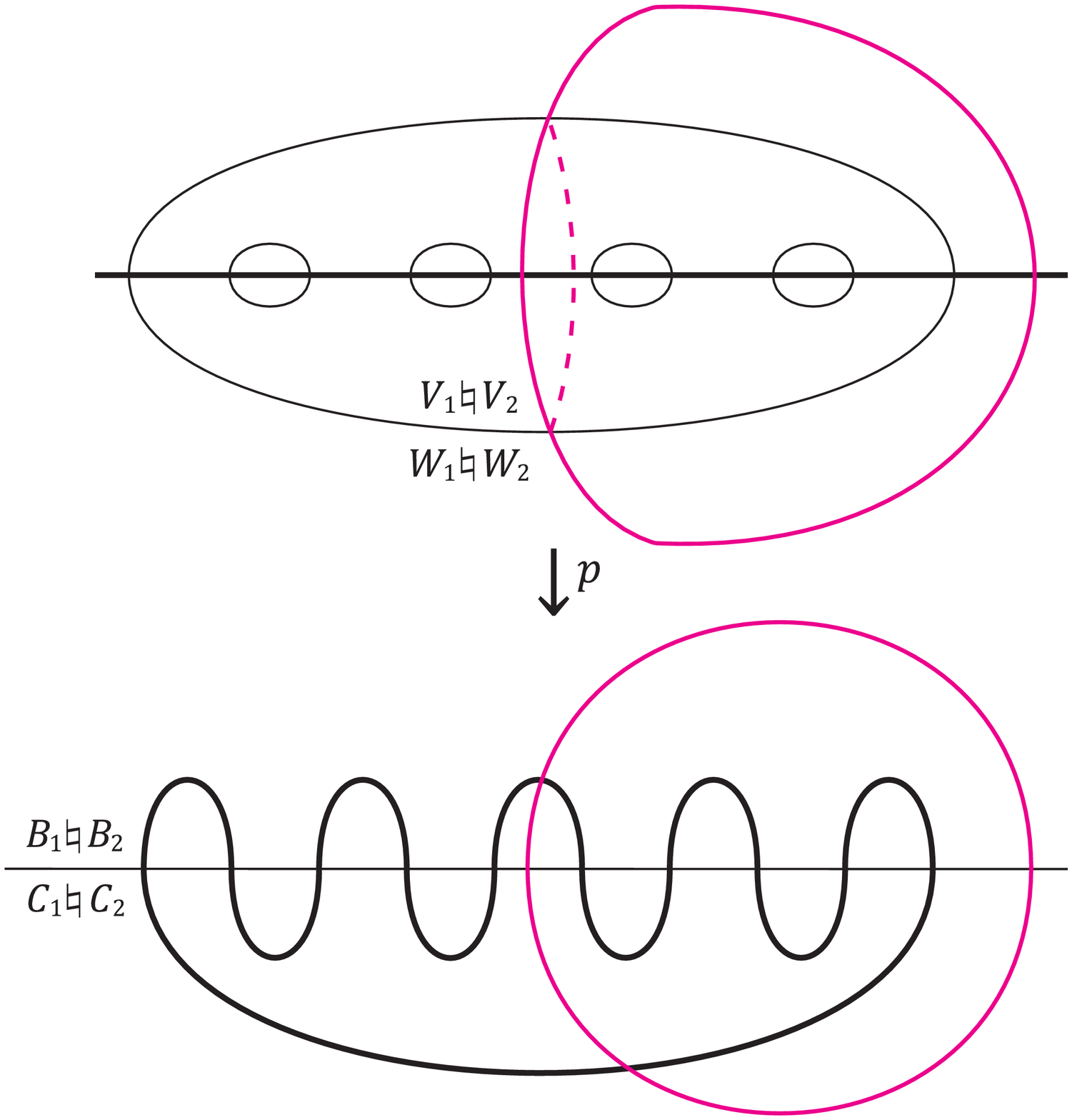}
\caption{A $2$-fold branched covering
$p : (V_1 \cup_{F_1} W_1) \# (V_2 \cup_{F_2} W_2) \to (B_1 \cup_{S_1} C_1) \# (B_2 \cup_{S_2} C_2)$.}\label{fig2}
\end{figure}

Since $V_i - p_i^{-1}(N_i)$ ($i=1,2$) $2$-fold branched covers $B_i - N_i$,
the handlebody $V_1 \natural V_2$ $2$-fold branched covers $B_1 \natural B_2$.
Similarly, since $W_i - p_i^{-1}(N_i)$ ($i=1,2$) $2$-fold branched covers $C_i - N_i$,
the handlebody $W_1 \natural W_2$ $2$-fold branched covers $C_1 \natural C_2$.
So $(V_1 \natural V_2) \cup_{F_1 \# F_2} (W_1 \natural W_2)$ $2$-fold branched covers
$(B_1 \natural B_2) \cup_{S_1 \# S_2} (C_1 \natural C_2)$.
See Figure \ref{fig2}.
Then by $(1)$ and $(2)$, we have the following lemma.

\begin{lemma}\label{lem1}
There is a $2$-fold branched covering
$p : (V_1 \cup_{F_1} W_1) \# (V_2 \cup_{F_2} W_2) \to (B_1 \cup_{S_1} C_1) \# (B_2 \cup_{S_2} C_2)$ .
\end{lemma}

In other words, by carefully choosing the $3$-balls,
a connected sum of $2$-fold branched coverings is a $2$-fold branched covering of a connected sum.

\section{A bridge version of Gordon's Conjecture}\label{sec6}

\begin{conjecture}[Gordon's Conjecture]\label{conj2}
The connected sum of two unstabilized Heegaard splittings is unstabilized.
\end{conjecture}

Gordon's Conjecture is proved by Bachman \cite{Bachman} and independently by Qiu and Scharlemann \cite{Qiu-Scharlemann}.
Bachman used the notion of critical surface.
The proof in \cite{Qiu-Scharlemann} is a constructive combinatorial proof.

We proposed a bridge version of Gordon's Conjecture in the introduction.
Compared to the case of Heegaard splittings, a difficulty that may arise by the presence of a knot is the following.
Suppose that a connected sum of two bridge positions is perturbed.
When we obtain, from a cancelling pair for the connected sum, subdisks in a summand,
there can be two subarcs that intersect at {\em two} points.
See Figure \ref{fig3}.
In Figure \ref{fig3}, $S$ is a bridge sphere and $P$ is a decomposing sphere for the connected sum.
On the other hand, it might be a helpful fact that every properly embedded disk in a $3$-ball is separating.

\begin{figure}[!hbt]
\includegraphics[width=14cm,clip]{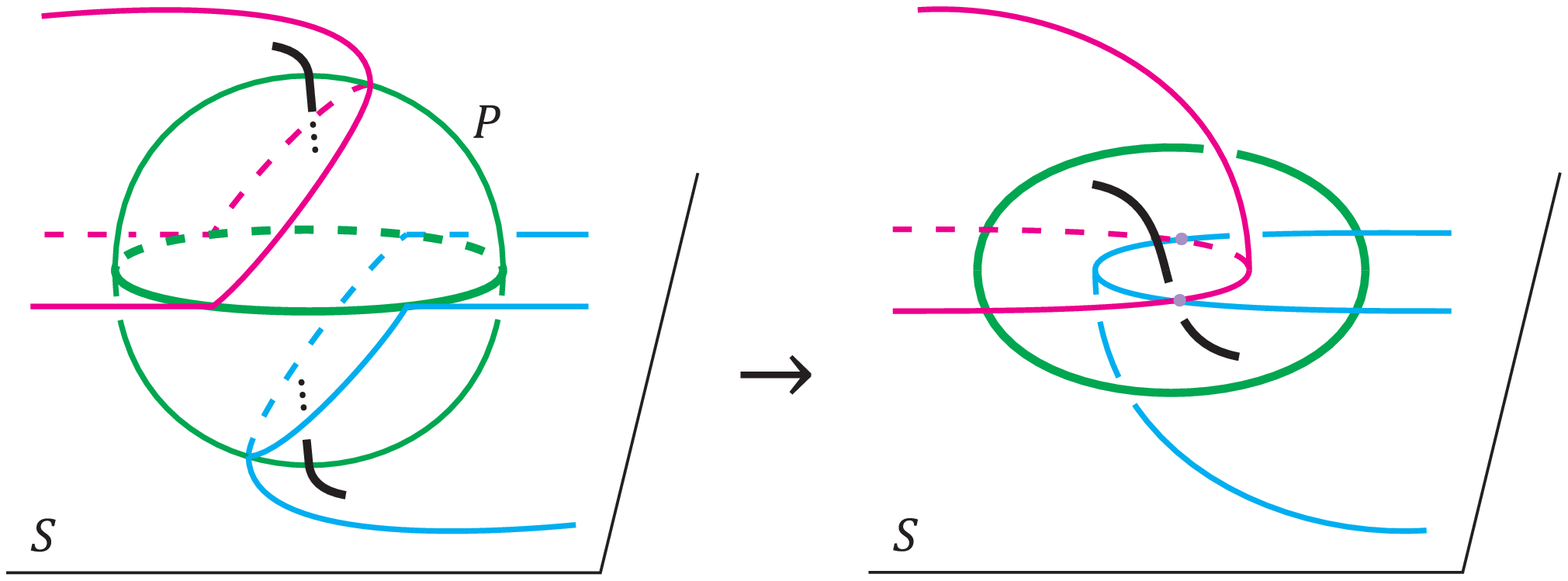}
\caption{From a cancelling pair for a connected sum, to subdisks in a summand.}\label{fig3}
\end{figure}

\section{Proof of Theorem \ref{thm1}}\label{sec7}

Let $K_1$ be a knot admitting an $n_1$-bridge position $B_1 \cup_{S_1} C_1$ with $n_1 > b(K_1)$
whose $2$-fold branched covering $V_1 \cup_{F_1} W_1$ is an unstabilized Heegaard splitting.

\begin{claim}\label{claim1}
There are infinitely many examples for $K_1$.
\end{claim}

\begin{proof}
There are infinitely many knots in \cite{JKOT}, each of which admits a $(2k+5)$-bridge position for any integer $k \ge 0$.
Let $K_1$ denote one of them, and $B_1 \cup_{S_1} C_1$ be a $(2k + 5)$-bridge position of $K_1$ with $2k + 5 > b(K_1)$.
It is shown in \cite{JKOT} that the $2$-fold branched covering $V_1 \cup_{F_1} W_1$ of $B_1 \cup_{S_1} C_1$ is strongly irreducible.
By Proposition \ref{prop1}, $V_1 \cup_{F_1} W_1$ is unstabilized.
\end{proof}

Let $K_2$ be a knot admitting an $n_2$-bridge position $B_2 \cup_{S_2} C_2$
whose $2$-fold branched covering $V_2 \cup_{F_2} W_2$ is unstabilized.
There are also infinitely many examples for $K_2$.
Then $(V_1 \cup_{F_1} W_1) \# (V_2 \cup_{F_2} W_2)$ is unstabilized by Gordon's Conjecture.
There exists a $2$-fold branched covering
$p : (V_1 \cup_{F_1} W_1) \# (V_2 \cup_{F_2} W_2) \to (B_1 \cup_{S_1} C_1) \# (B_2 \cup_{S_2} C_2)$ by Lemma \ref{lem1}.
By Proposition \ref{prop2}, $(B_1 \cup_{S_1} C_1) \# (B_2 \cup_{S_2} C_2)$ is unperturbed.
It is weakly reducible because it is obtained by a connected sum.
The bridge number $b(K_1 \# K_2)$ is $b(K_1) + b(K_2) - 1$ and
$(B_1 \cup_{S_1} C_1) \# (B_2 \cup_{S_2} C_2)$ is an $(n_1 + n_2 - 1)$-bridge position of $K_1 \# K_2$, where
$n_1 + n_2 - 1 > b(K_1) + b(K_2) - 1$.



\end{document}